\documentclass[11pt]{amsart}
\usepackage{amsmath,amssymb}
\usepackage{verbatim}
\usepackage{color}
\usepackage{hyperref}

\begin{document}

\newtheorem{thm}{Theorem}[section]
\newtheorem{theorem}{Theorem}[section]
\newtheorem{lem}[thm]{Lemma}
\newtheorem{lemma}[thm]{Lemma}
\newtheorem{prop}[thm]{Proposition}
\newtheorem{proposition}[thm]{Proposition}
\newtheorem{corollary}[thm]{Corollary}
\newtheorem{definition}[thm]{Definition}
\newtheorem{remark}[thm]{Remark}
\newtheorem{conjecture}[theorem]{Conjecture}

\numberwithin{equation}{section}

\newcommand{\Z}{{\mathbb Z}} 
\newcommand{\Q}{{\mathbb Q}}
\newcommand{\R}{{\mathbb R}}
\newcommand{\C}{{\mathbb C}}
\newcommand{\N}{{\mathbb N}}
\newcommand{\FF}{{\mathbb F}}
\newcommand{\fq}{\mathbb{F}_q}
\newcommand{\X}{{\mathbb {X}}}
\newcommand{\rmk}[1]{\footnote{{\bf Comment:} #1}}

\newcommand{\bfA}{{\boldsymbol{A}}}
\newcommand{\bfY}{{\boldsymbol{Y}}}
\newcommand{\bfX}{{\boldsymbol{X}}}
\newcommand{\bfZ}{{\boldsymbol{Z}}}
\newcommand{\bfa}{{\boldsymbol{a}}}
\newcommand{\bfy}{{\boldsymbol{y}}}
\newcommand{\bfx}{{\boldsymbol{x}}}
\newcommand{\bfz}{{\boldsymbol{z}}}
\newcommand{\F}{\mathcal{F}}
\newcommand{\Gal}{\mathrm{Gal}}
\newcommand{\Fr}{\mathrm{Fr}}
\newcommand{\Hom}{\mathrm{Hom}}
\newcommand{\GL}{\mathrm{GL}}

\renewcommand{\mod}{\;\operatorname{mod}}
\newcommand{\ord}{\operatorname{ord}}
\newcommand{\TT}{\mathbb{T}}
\renewcommand{\i}{{\mathrm{i}}}
\renewcommand{\d}{{\mathrm{d}}}
\renewcommand{\^}{\widehat}
\newcommand{\HH}{\mathbb H}
\newcommand{\Vol}{\operatorname{vol}}
\newcommand{\area}{\operatorname{area}}
\newcommand{\tr}{\operatorname{tr}}
\newcommand{\norm}{\mathcal N} 
\newcommand{\intinf}{\int_{-\infty}^\infty}
\newcommand{\ave}[1]{\left\langle#1\right\rangle} 
\newcommand{\Var}{\operatorname{Var}}
\newcommand{\Prob}{\operatorname{Prob}}
\newcommand{\sym}{\operatorname{Sym}}
\newcommand{\disc}{\operatorname{disc}}
\newcommand{\CA}{{\mathcal C}_A}
\newcommand{\cond}{\operatorname{cond}} 
\newcommand{\lcm}{\operatorname{lcm}}
\newcommand{\Kl}{\operatorname{Kl}} 
\newcommand{\leg}[2]{\left( \frac{#1}{#2} \right)}  
\newcommand{\Li}{\operatorname{Li}}

\newcommand{\sumstar}{\sideset \and^{*} \to \sum}

\newcommand{\LL}{\mathcal L} 
\newcommand{\sumf}{\sum^\flat}
\newcommand{\Hgev}{\mathcal H_{2g+2,q}}
\newcommand{\USp}{\operatorname{USp}}
\newcommand{\conv}{*}
\newcommand{\dist} {\operatorname{dist}}
\newcommand{\CF}{c_0} 
\newcommand{\kerp}{\mathcal K}

\newcommand{\Cov}{\operatorname{cov}}
\newcommand{\Sym}{\operatorname{Sym}}

\newcommand{\ES}{\mathcal S} 
\newcommand{\EN}{\mathcal N} 
\newcommand{\EM}{\mathcal M} 
\newcommand{\Sc}{\operatorname{Sc}} 
\newcommand{\Ht}{\operatorname{Ht}}

\newcommand{\E}{\operatorname{E}} 
\newcommand{\sign}{\operatorname{sign}} 

\newcommand{\divid}{d} 

\newcommand{\h}{\mathbb{H}_{2g+1}}
\newcommand{\p}{\mathbb{P}_{2g+1}}
\newcommand{\f}{\mathbb{F}_{q}[T]}
\newcommand{\z}{\zeta_A}
\newcommand{\lo}{\log_q}
\newcommand{\x}{\chi}
\newcommand{\xx}{\mathcal{X}}
\newcommand{\lL}{\mathcal{L}}
\newcommand{\e}{\varepsilon}
\newcommand{\w}{\omega}
\newcommand{\pp}{\text{\textbf{P}}}

\title[THE CLASS NUMBERS OF PRIMES]
{The Moments and statistical Distribution of Class number of Primes over Function Fields}

\author{Julio Andrade}
\address{Department of Mathematics, University of Exeter, Exeter, EX4 4QF, United Kingdom}
\email{j.c.andrade@exeter.ac.uk}

\author{ASMAA sHAMESALDEEN}
\address{Department of Mathematics, University of Exeter, Exeter, EX4 4QF, United Kingdom}
\email{as1029@exeter.ac.uk}

\thanks{The first author is grateful to the Leverhulme Trust (RPG-2017-320) for
the support through the research project grant “Moments of L-functions
in Function Fields and Random Matrix Theory”. The second author was
supported by a Ph.D. scholarship from the government of Kuwait.} 

\subjclass[2010]{Primary 11M38; Secondary 11M06, 11G20, 11M50, 14G10}
\keywords{Mean values of $L$--functions; finite fields; function fields; class numbers}

\begin{abstract}
We investigate the moment and the distribution of $L(1,\x_P),$ where $\x_P$ varies over quadratic characters associated to irreducible polynomials $P$ of degree $2g+1$ over $\mathbb{F}_q[T]$ as $g\to\infty$. In the first part of the paper we compute the integral moments of the class number $h_{P}$ associated to quadratic function fields with prime discriminants $P$ and this is done by adapting to the function field setting some of the previous results carried out by Nagoshi in the number field setting. In the second part of the paper we compute the complex moments of of $L(1,\x_P)$ in large uniform range and investigate the statistical distribution of the class numbers by introducing a certain random Euler product. The second part of the paper is based on recent results carried out by Lumley when dealing with square-free polynomials.
\end{abstract}
\date{\today}

\maketitle

\section{Introduction}\label{into}

Gauss in his \textit{Disquistiones Arthmeticae} \cite{Gauss}, presented two conjectures concerning the average values of the class numbers $h_D$ associated with binary quadratic forms $ax^2+2bxy+cy^2,$ where $a,b$ and $c$ are integers, and $D=4(b^2-ac)$ is the discriminant of the  binary quadratic forms $ax^2+2bxy+cy^2.$  Gauss conjectured that
 
\begin{equation*}
\sum_{\substack{ 0<-D\le X\\ D \equiv 0 \mod 4}} h_D \sim \frac{\pi}{42 \zeta(3)}X^{3/2} \; \text{  and  } \; \sum_{\substack{ 0<D\le X\\ D \equiv 0 \mod 4}} h_D \log \varepsilon_D \sim \frac{\pi^2}{42 \zeta(3)}X^{3/2}
\end{equation*}
as $X\to\infty,$ where $\zeta(3)=\sum_{n=1}^\infty n^{-3}$. Later on, these two conjectures where proved by Lipschitz \cite{lip} and Siegel \cite{sig}. 
\newline

Let $d$ denote a fundamental discriminant and let $\Q(\sqrt{d})$ be the quadratic field with discriminant $d$ and $h_d$ represent the class number of this field. It is a fundamental problem in number theory to understand the distribution value of the size of the class group for a given field. It is not a surprise then that describing the extreme values of $h_{d}$ and their distribution values have been vastly investigated. For example, Granville and Soundararajan \cite{c&s}, and Dahl and Lamzouri \cite{d&l} make use of a random model to study the moments of the class number through the use of Dirichlet's formula that connects $h_{d}$ with the value of the Dirichlet quadratic $L$-function at $s=1$, i.e., with $L(1,\chi_{d})$. Following the work of Granville and Soundararajan, Nagoshi in \cite{nagoshi} established asymptotic formulas for all the moments of $L(1,\x_p)$ with $\x_p$ denoting the real character modulo $p$ given by the Legendre symbol $\left(\frac{\cdot}{p}\right)$, where $p$ is an odd prime.
\newline

Let $d_k(n), k\in\N$ be the generalized $k$-th divisor function, define 

\begin{equation}\label{abk}
\widetilde{a}_k := \sum_{m=1}^\infty \frac{d_k(m^2)}{m^2}\in \R,
\end{equation}
which is convergent by the bound $d_k(n)\ll_{k,\varepsilon}n^\varepsilon,$ for any $\varepsilon>0.$ Nagoshi proved the following. 

\begin{theorem}\label{nagoshithm1} (Nagoshi's Theorem)
	Let $v$ be the integer $1$ or $3$. Let $k\in\N$ and $X\ge 5.$ Then 
	
	\begin{equation*}
	\sum_{\substack{p\le X\\ p\equiv v \mod 4}} \left(\log p\right) L\left(1,\chi_p\right)^k=  \frac{\widetilde{a}_k}{2} X + O_{k,\delta}\left(\frac{X}{(\log X)^{2-\delta}}\right) \text{ for any } \delta>0,
	\end{equation*}
	where the implied constant is effectively computable.
\end{theorem}

As a consequence of the above theorem, Nagoshi established the following asymptotic formulas for all the moment of the class number $h_p,$

\begin{equation*}\label{nagoshithm21}
\sum_{\substack{p\le X\\ p\equiv 3 \mod 4}}h(-p)^k \sim \frac{\widetilde{a}_k}{\pi^k(k+2)}\frac{X^{1+k/2}}{\log X} \left( 1 + \frac{2}{(k+2)}\frac{1}{\log X}\right),
\end{equation*}
and
\begin{equation*}\label{nagoshithm22}
\sum_{\substack{p\le X\\ p\equiv 1 \mod 4}}\left(h(p)\log\varepsilon(p)\right)^k \\
\sim\frac{\widetilde{a}_k}{2^k(k+2)}\frac{X^{1+k/2}}{\log X} \left( 1 + \frac{2}{(k+2)}\frac{1}{\log X}\right).
\end{equation*}
\newline
%

Moreover, Nagoshi investigated the distribution of the class numbers of quadratic fields with prime discriminant. He compared the distribution of values of $L(1,\x_p)$ with the distribution of random Euler products $L(1,W_p)=\prod_p \left(1-W_p(\w)/p\right)^{-1}$ where the $W_p(\w)$'s are independent random variables $\pm 1$ with suitable probabilities (see \cite{lau} and \cite{c&s}). 
\newline

Let $\left\{\X\right\}$ be a sequence of independent random variables on a probability space and $E[\X]$ be its expected value, $G$ and $\widetilde{G}$ are the  distribution functions of the random variable $L(1,\X)$ and $\log L(1,\X),$ respectively, that is, for $x\in\R$

\begin{equation*}
G(x) := \pp\left(\left\{ L(1,\X)\le x\right\}\right)\text{ and }
\widetilde{G}(x) := \pp\left(\left\{\log L(1,\X)\le x\right\}\right). 
\end{equation*} 
With this notation, Nagoshi proved the following.

\begin{theorem}(Nagoshi's Distribution Theorem)\label{nagoshiDistrb}
	For each $x\in\R,$ we have 
	
	\begin{equation*}
	\lim_{N\to\infty} \frac{\#\left\{p\le N \mid p \equiv 3 \mod 4, h(-p)\le \pi^{-1} \sqrt{p} e^x \right\} }{\#\left\{p\le N \mid p \equiv 3 \mod 4 \right\}} = G(e^x) =\widetilde{G}(x),
	\end{equation*}
	and 
	
	\begin{equation*}
	\lim_{N\to\infty} \frac{\#\left\{p\le N \mid p \equiv 1 \mod 4, h(p)\log\varepsilon(p) \le 2^{-1} \sqrt{p} e^x \right\} }{\#\left\{p\le N \mid p \equiv 1 \mod 4 \right\}} = G(e^x) =\widetilde{G}(x).
	\end{equation*}
	The distribution function $G$ is strictly increasing on $(0,\infty),$ and $\widetilde{G}$ is strictly increasing on $\R$. The characteristic function of $G$ has the form
	
	\begin{equation*}
	1+\sum_{k=1}^\infty \frac{\widetilde{a}_k}{k!}\left(\i t\right)^k, \;\; t\in\R,
	\end{equation*}
	where the numbers $\widetilde{a}_k$ are as in (\ref{abk}). The characteristic function $E\left[L(1,\cdot)^{\i t}\right]$ of $\widetilde{G}$ has the form 
	
	\begin{equation*}
	\prod_{p} \left(\frac{1}{2} \left(1-\frac{1}{p}\right)^{-\i t}+ \frac{1}{2} \left(1+\frac{1}{p}\right)^{-\i t}\right), \;\; t\in \R,
	\end{equation*}
	and satisfies
	
	\begin{equation}
	E\left[L(1,\cdot)^{\i t}\right]\ll\exp\left(-c \frac{|t|}{\ln(2+|t|)}\right) \text{ for all } t\in\R
	\end{equation}
	with absolute constant $c>0$. The distribution function $\widetilde{G}$ has a density $g$. Further, $\widetilde{G}$ and $g$ are infinitely differentiable.
\end{theorem}

In the first part of this paper we prove the function field analogue of Nagoshi's results and study the class number, denoted as $h_P$, over function field, $\mathbb{F}_{q}(T)$ with $q \equiv 1 \left(\mod 4\right)$ and $P$ is a monic irreducible polynomial in $\f$.
\newline

In 1992, Hoffstein and Rosen \cite{H&R} investigated the average value of the class number $h_D$ when the average is taken over all monic polynomial of a fixed degree, they showed, for $M$ odd and positive, that
\begin{equation*}
\frac{1}{q^M} \sum_{\substack{ D \text{ monic}\\ \deg(D)=M}} h_D= \frac{\z(2)}{\z(3)}q^{(M-1)/2}-q^{-1},
\end{equation*}
where $\z(s)$ is the Riemann zeta function over $\f$. Note that Hoffstein and Rosen result is the directly comparable to Gauss's conjectures.
\newline

In a recent paper, Andrade \cite{jca} established an asymptotic formula for the mean value of the class number $h_{D}$ over function fields when the average is taken over $\h$, the set of all monic, square free polynomials of degree $2g+1$ in $\f$. Andrade proved that, as $g\to\infty$ we have

\begin{equation*}
\frac{1}{|\h|} \sum_{D\in\h}h_D \sim \z(2) q^g \prod_{\substack{P \text{irreducible}}} \left(1-\frac{1}{\left(|P|+1\right)|P|^2}\right).
\end{equation*}

In a more recent paper, Lumley \cite{Lum} investigated the distribution of $L(1,\x_D)$ for $D\in\mathbb{H}_n$ as $n\to \infty$. She computed large complex moments of the associated $L(1,\x_D)$ using the technique of random models that has been used successfully in the study of quadratic number fields. Lumley proved that we can express the complex moments of $L(1,\x_D)$ as follows.

\begin{theorem} \label{lumleythm1}
	Let $n$ a positive integer, and $z\in\C$ be such that $|z|\le \frac{n}{260\lo (n)\log\lo(n)}.$ Then 
	\begin{equation*}
	\frac{1}{\left|\mathbb{H}_n\right|} \sum_{D\in\mathbb{H}_n} L\left(1,\x_D\right)^z= \sum_{f \text{ monic}} \frac{d_k(f^2)}{|f|^2} \prod_{P\mid f} \left(1+\frac{1}{|P|}\right)^{-1} \left(1+O\left(\frac{1}{n^{11}}\right)\right).
	\end{equation*}
\end{theorem}

As consequence of the above theorem, Lumley stated that if we specialize $n$ to be $n=2g+1$ and $2g+2$ and letting the genus $g\to\infty$ we have the following results.

\begin{corollary}\label{lumleycor1}
	Let $z\in\C$ be such that $|z|\le \frac{g}{130 \lo(g)\log\lo(g)}.$ Then 
	\begin{equation*}
	\begin{split}
	\frac{1}{\left|\mathbb{H}_{2g+1}\right|}& \sum_{D\in\mathbb{H}_{2g+1}} h_D^z\\
	& = q^{gz} \sum_{f \text{ monic}} \frac{d_k(f^2)}{|f|^2} \prod_{P\mid f} \left(1+\frac{1}{|P|}\right)^{-1} \left(1+O\left(\frac{1}{g^{11}}\right)\right).
	\end{split}
	\end{equation*}
\end{corollary}
	
	\begin{corollary}\label{lumleycor2}
		Let $z\in\C$ be such that $|z|\le \frac{g}{130 \lo(g)\log\lo(g)}.$ Then 
		\begin{equation*}
		\begin{split}
		\frac{1}{\left|\mathbb{H}_{2g+1}\right|} & \sum_{D\in\mathbb{H}_{2g+1}} \left(h_DR_D\right)^z\\
		& = \left(\frac{q^{g+1}}{q-1}\right)^{z} \sum_{f \text{ monic}} \frac{d_k(f^2)}{|f|^2} \prod_{P\mid f} \left(1+\frac{1}{|P|}\right)^{-1} \left(1+O\left(\frac{1}{g^{11}}\right)\right),
		\end{split}
		\end{equation*}
where $R_{D}$ is the regulator of the associated quadratic function field.
\end{corollary}

In the second part of this paper we will adapt Lumley's result and investigate the complex moment of $L(1,\x_P)$ in a large uniform range, where $\x_P$ varies over quadratic characters associated to irreducible polynomials $P$ of degree $n$ over $\mathbb{F}_q$ as $n\to\infty.$

 \section{Preparations}

Before we state the main results of this paper we first introduce some
notation and auxiliary results.
Let $\mathbb{F}_q$ be a finite field with $q$ elements where $q$ is a prime power. We denote by $A = \f$ the polynomial ring over $\mathbb{F}_q$ and the norm of a polynomial $f\in A$ is defined to be $|f| = q^{\deg(f)}$. 
\newline

Let $\mathbb{P}_{n}$ to denote the set of all monic irreducible polynomials in $\mathbb{F}_{q}[T]$ of degree $n$ and let $\chi_{P}(f)$ to denote the quadratic character associated to a monic irreducible polynomial $P$, the value of the character is defined in terms of the Legendre symbol for polynomials over finite fields. The associated Dirichlet $L$-function is defined in the usual way as

$$L(s,\chi_{P})=\sum_{f \ \text{monic}}\frac{\chi_{P}(f)}{|f|^{s}}.$$

For the remainder of this paper the following notations will be fixed. Let  $\log$ denotes the logarithm in the base $q$, $\ln$ is the natural logarithm and $\log_j$ ( respectively $\ln_j$) represents the $j$-fold iterated logarithm. Finally, let $P$ be an irreducible (prime) polynomial in $A$. 
\newline

Our first auxiliary result is the following.

%
%
%
%
%
%
%

%
%
%
%
%

\begin{proposition}(``Approximate" functional equation)
Let $P\in\mathbb{P}_{2g+1}$, then we have that
 
	\begin{equation}\label{aproxL(1,xp)k}
	\begin{split}
	L(1,\x_P)^k =& \sum_{\substack{f_1 \text{ monic} \\ \deg(f_1)\le kg}} \frac{\x_P(f_1)d_k(f_1)}{|f_1|} + q^{-kg} \sum_{\substack{f_2 \text{ monic} \\ \deg(f_2)\le kg-1}}  \x_P(f_2) d_k(f_2).
	\end{split}
	\end{equation}	
%
\end{proposition}
\begin{proof}
	Recall that,
	
	\begin{equation*}
	\begin{split}
	L(s,\x_P)= & \sum_{\substack{f \text{ monic}}} \frac{\x_P(f)}{|f|^s}\\
	=& \sum_{n=0}^\infty q^{-sn} \sum_{\substack{f \text{ monic} \\ \deg(f)=n}}\x_P(f).
	\end{split}
	\end{equation*}
	Therefore, we have
	
	\begin{equation}\label{Lk1}
	\begin{split}
	L(s,\x_P)^k
	=& \sum_{n=0}^\infty  q^{-sn} \sum_{\substack{f \text{ monic} \\ \deg(f)=n}}\x_P(f)d_k(f),
	\end{split}
	\end{equation}
	where $d_k(f)$ is the number of ways that $f$ can be expressed as a product of $k$ monic (taking order into account).
	Since $P\in\mathbb{P}_{2g+1}$, we have that $L(s,\x_P)$ is a polynomial of degree $2g$ and therefore 
	
	\begin{equation}\label{Lk2}
	L(s,\x_P)=L_{C_P}(u).
	\end{equation}
	Moreover,
	
	\begin{equation*}
	L_{C_P}(u)=\left(qu^2\right)^g L_{C_P}\left(\frac{1}{qu}\right),
	\end{equation*}
	and so
	
	\begin{equation*}
	L_{C_P}(u)^k=\left(qu^2\right)^{kg} L_{C_P}\left(\frac{1}{qu}\right)^k.
	\end{equation*}
	Let $L_{C_P}(u)^k=\sum_{n=0}^{kg}a_nu^n,$ then we have
	
	\begin{equation*}
	\begin{split}
	\sum_{n=0}^{2kg}a_nu^n 
	=& \sum_{r=0}^{2kg}a_{2kg-r} q^{r-kg}u^{r}.
	\end{split}
	\end{equation*}
	Comparing the coefficients we find that $a_n= a_{2kg-r} q^{r-kg}$ and we can write 
	
	\begin{equation}\label{L(1,xP)k}
	\begin{split}
	L(s,\x_P)^k 
	=&  \sum_{n=0}^{kg} a_n u^n + (qu^2)^{kg}\sum_{m=0}^{kg-1}a_{m} q^{-m} u^{-m}.
	\end{split}
	\end{equation}
	From (\ref{Lk1}) and (\ref{Lk2}) we can write the coefficients $a_n$ as
	
	\begin{equation*}
	a_n=\sum_{\substack{f \text{ monic} \\ \deg(f)=n}} \x_P(f)d_k(f)
	\end{equation*}
and this proves the result.
\end{proof}

The next result is the well-known prime polynomial theorem.

\begin{theorem}(Prime Polynomial Theorem)\label{PNT}\\
	The number of monic irreducible polynomials in $A=\mathbb{F}_q[T]$ of degree $n$ is 
	
	\begin{equation*}
	\pi_ A(n)=\frac{q^n}{n}+O\Big(\frac{q^{n/2}}{n}\Big).
	\end{equation*}
\end{theorem}

\begin{lemma}\label{dl}
	Let $f$ be a monic polynomial in $\f$, $k\ge 2$, and $d_k(f)$ be the $k$-fold divisor function. Then
	
	\begin{equation*}\label{dl2}
	\begin{split}
	\sum_{\substack{f \text{ monic}\\ \deg(f)=n}}d_k(f) &=  \frac{1}{(k-1)!} q^n n^{k-1} + O(q^nn^{k-2}).
	\end{split}
	\end{equation*}
%
\end{lemma}

For Lemma \ref{dl} see Lemma 2.2 in \cite{and-bary-rudnik} 
. Our next result is quoted from Rosen \cite[Chapter 17]{Rosen}.

\begin{lemma}\label{Rosentheorem17.4}
	Let $A^+$ be the set of monic polynomials in $\f$ and 
	
	\begin{equation*}
	B=\left\{s\in\mathbb{C}:-\frac{\pi \i}{\ln q}\le \mathfrak{I}(s) \le \frac{\pi \i}{\ln q}\right\}.
	\end{equation*} 
	Let $f:A^+\to \mathbb{C}$, and $\zeta_f(s)$ be the corresponding Dirichlet series. Suppose this series converges absolutely in the region $\mathfrak{R}(s)>1$ and is holomorphic in the region $\{s\in B:\mathfrak{R}(s)=1\}$ except for a single pole of order $r$ at $s=1$. Let $\alpha=\lim_{s\to 1} (s-1)^r\zeta_f(s).$ Then there is a $\delta < 1$ and constant $c_{-i}$ with $1\le i\le r$ such that 
	
	\begin{equation*}
	\sum_{\deg(D)=n} f(D)=q^n \left(\sum_{i=1}^r c_{-i} \binom{n+i-1}{i-1} (-q)^i\right)+O\left(q^{\delta n}\right).
	\end{equation*}
	The sum in parenthesis is a polynomial in $n$ of degree $r-1$ with leading term 
	
	\begin{equation*}
	\frac{(\ln q)^r}{(r-1)!}\alpha n^{r-1}.
	\end{equation*}
\end{lemma}

\begin{lemma} \label{Dlk}
	Let $f$ be a monic polynomial in $\f$, and $d(f)$ be the number of monic divisors of $f$. Let $\zeta_{d_k}(s)$ be the corresponding Dirichlet series. Then $\zeta_{d_k}(s)$ converges absolutely in the region $\mathfrak{R}(s)>1$ and holomorphic in the region $\left\{s\in B, \mathfrak{R}(s)=1\right\}$ except for a pole of order $k(k+1)/2$ at $s=1$. Let $\rho_k=\lim_{s\to\infty}(s-1)^{\frac{k(k+1)}{2}} \zeta_{d_k}(s)$, then for a fixed $\epsilon>0$ and constants $c_{-i}$ with $1\le i\le\frac{k(k+1)}{2}$ we have
	
	\begin{equation}\label{Dl1}
	\sum_{\substack{f \text{ monic}\\ \deg(f)=n}}d_k(f^2)= q^n \left(\sum_{i=1}^{\frac{k(k+1)}{2}} c_{-i} \binom{n+i-1}{i-1} \left(-q\right)^i\right)+O\left(q^{\epsilon n}\right).
	\end{equation}
	The sum is parenthesis is a polynomial in $n$ of degree $\frac{k(k+1)}{2}-1$ with leading term 
	
	\begin{equation}\label{leading}
	\frac{A_k(1)}{\left(\frac{k(k+1)}{2}-1\right)!} n^{\frac{k(k+1)}{2}-1},
	\end{equation}
	where the definition of $A_{k}(s)$ is presented in the proof of this lemma. When $k=2$, we can write 
	
	\begin{equation}\label{Dl2}
	\sum_{\substack{f \text{ monic}\\ \deg(f)=n}}d(f^2)=\{1+\frac{1}{2} (3+q^{-1}) n +\frac{1}{2} (1-q^{-1}) n^2 \} q^n.
	\end{equation}
\end{lemma}

\begin{proof}	
	Let 	

	\begin{equation*}
	\zeta_f(s)= \sum_{f \text{ monic}} \frac{d_k(f^2)}{|f|^s}
	\end{equation*}
	be the zeta function associated to $d_k(f^2).$ Recall that 

	\begin{equation}\label{dkp}
	\begin{split}
	d_k(P^{r}) 
	=& \frac{(k+r-1)!}{(k-1)! r!},
	\end{split}
	\end{equation}
	Then the zeta function can be written as	

	\begin{equation*}
	\begin{split}
	\zeta_f(s) 
	& =  \prod_{\substack{P \text{ monic} \\ \text{irreducible}}} \left(1+ \sum_{n=1}^\infty \frac{1}{|P|^{sn}} \frac{(k+2n-1)!}{(k-1)! (2n)!}\right)\\
	& =  \prod_{\substack{P \text{ monic} \\ \text{irreducible}}}  \frac{1}{2} \left(1-|P|^{-s}\right)^{-k} \left(\left(1-|P|^{-\frac{s}{2}}\right)^k+\left(|P|^{-\frac{s}{2}}+1\right)^k\right) \\
	& = \left(\z(s)\right)^{\frac{k(k+1)}{2}} \prod_{\substack{P \text{ monic} \\ \text{irreducible}}}   \left( 1-|P|^{-s}\right)^{\frac{k(k-1)}{2}} \sum_{i=0}^{\left[k/2\right]} \binom{k}{2i} |P|^{-is} \\
	& =  \left(\z(s)\right)^{\frac{k(k+1)}{2}} A_k(s).
	\end{split}
	\end{equation*}
	From the definitions of $\z(s)$ and $A_k(s)$ the sum converges absolutely for $\mathfrak{R}(s)>1$, is holomorphic on the disc $\left\{u=q^{-s}\in\mathbb{C}:|u|\le q^{-\delta}\right\}$ for some $\delta<1$, and $\zeta_f(s)$ has a pole of order $k(k+1)/2$ at $s=1$. Applying Lemma \ref{Rosentheorem17.4} equation (\ref{Dl1}) follows. Since we have 	
	
	\begin{equation*}
	\begin{split}
	\rho_k & = \lim_{s\to 1}(s-1)^{k(k+1)/2} \left(\z(s)\right)^{\frac{k(k+1)}{2}} A_k(s)\\
	&= \frac{1}{\left(\ln q\right)^{k(k+1)/2}} A_k(1),
	\end{split}
	\end{equation*}
	then by applying the formula for the leading term of the polynomial in parenthesis given in the statement of Lemma \ref{Rosentheorem17.4}, we get equation (\ref{leading}). For (\ref{Dl2}) see Lemma 5.1 in \cite{a&S&H}.
\end{proof}

The next result is a bound for non-trivial character sums.

\begin{proposition}\label{chi}
	If $f\in \f$ is monic and not a perfect square, with $\deg(f)>0$ then we have that 	

	\begin{equation*}
	\Bigg|\sum_{\substack{P \text{ irreducible}\\ \deg(P)=n}}\left(\frac{f}{P}\right)\Bigg|\ll\frac{q^{\frac{n}{2}}}{n} \deg(f).
	\end{equation*}	
\end{proposition}

For the proposition above see page 87 in \cite{Rudnick}. The next result follows from Proposition \ref{chi} and Lemma \ref{dl}.

%
%

\begin{lemma}\label{NkBox}\label{NBox}
	Let $f$ be a monic polynomial in $\f$ of degree $n,$ then if $f$ is not a perfect-square we have
	
	\begin{equation*}
	\sum_{P\in\p}\, \sum_{\substack{\deg(f)=n\\ f\neq\Box}}\chi_p(f)d_k(f)\ll  \frac{\sqrt{|P|}}{\log|P|} q^n n^k.
	\end{equation*}
\end{lemma}
%
With the previous results in hands we can establish the following result.

\begin{lemma}\label{nkbox}
	Let $f$ be a monic polynomial in $\f$. If $f$ is not a perfect-square we have
	
	\begin{equation*}\label{nkbox1}
	\text{(1) } \log|P| \sum_{P\in\p}  \sum_{\substack{f \text{ monic}\\ \deg(f)\le kg \\ f\neq\Box}} \frac{d_k(f)}{|f|}\x_P(f)  = O\left(|P|^{\frac{1}{2}} \left(\log|P|\right)^{k+1}\right),
	\end{equation*}
	and
	
		\begin{equation*}\label{nkbox2}
	\text{(2) } q^{-kg} \log|P|\sum_{P\in\p} \sum_{\substack{f \text{ monic}\\ \deg(f)\le kg-1 \\ f\neq\Box}} \x_P(f) d_k(f) = O\left(|P|^{\frac{1}{2}} \left(\log|P|\right)^{k}\right).
	\end{equation*}
\end{lemma}

\begin{lemma}[Merten's Theorem \cite{Ros1}]\label{Mertens}
	Let $P\in\f$ be monic irreducible polynomial. Then, we have	
	\begin{equation*}
	\prod_{\substack{P \text{ irreducible} \\ \deg(P)\leqslant X}} \left(1-\frac{1}{|P|}\right)^{-1} = e^\gamma X+O(1),
	\end{equation*}
	where $\gamma$ is the Euler constant.
\end{lemma} 

\subsection{The Random Euler Product}$\color{white}bnkjnd$\label{random product}\\
We present in this section the probabilistic model that we will use when studying $L(s,\chi_{P})$. Let $\left\{W_P \mid P \text{ prime}\right\}$ be a sequence on independent random variables on a probability space $\left(\Omega,\mathcal{F},\pp\right)$ such that for each $\w\in\Omega$ with $W_P(\w)=\pm 1$ has probability $1/2$, note that by Theorem 5.3 in \cite{Billingsley} such probability exists.
\newline

We define the random Euler product $L\left(1,\w\right),\w\in\Omega,$ by

\begin{equation}\label{lomega}
\begin{split}
L(1,\w):&= \sum_{\w \in \Omega} \frac{W_P(\w)}{|\w|}\\
& = \prod_{\substack{P \text{ monic} \\ \text{irreducible}}} \left(1- \frac{W_P(\w)}{|P|}\right)^{-1}\\
&= \prod_{\substack{P \text{ monic} \\ \text{irreducible}}} L_P(1,\omega),
\end{split}
\end{equation}
which converges almost surely (see \cite{c&s}, \cite{lau} and \cite{Lum}).  Let $E[Y]$ be the expected value of a random variable $Y$ on $\Omega$ that is defined by 

\begin{equation*}
E\left[Y\right]= \sum_{i=1}^k x_i \mathbf{P}(x_i),
\end{equation*}
where $Y(\omega_i)=x_i, \omega_i\in\Omega, i=1,\cdots,k$ and $\mathbf{P}$ is the probability distribution of the random variable. Since for each prime $P$ $E\left[W_P/|P|\right] = 0,$ we have

\begin{equation*}
\begin{split}
\sum_{\substack{P \text{ monic}\\ \text{irreducible}}} E\left[\left|\frac{W_P}{|P|}\right|^2\right] 
&= \sum_{\substack{P \text{ monic}\\ \text{irreducible}}} \frac{1}{|P|^2} 
<\infty,
\end{split}
\end{equation*}
and 

\begin{equation*}
L_P(1,\w) = 1+\sum_{n=1}^\infty \frac{W_P(\w)^n}{|P|^n},
\end{equation*}
which converges for almost all $\w\in\Omega,$ (see Theorem 1.7 in \cite{lau}). Moreover, $L(1,\w)>0$ for almost all $\w\in\Omega$.  

\begin{lemma}\label{E=EP}
	Let $k>0$, then the infinite product $\prod_{P} E\left[L_P(1,.)^k\right]$ is convergent, the random variable $L(1,\w)^k, \w\in\Omega$ is integrable, and we have 	

	\begin{equation*}
	\prod_{\substack{P \text{ monic} \\ \text{irreducible}}} E\left[L_P(1,.)^k\right]=E\left[L(1,.)^k\right]
	\end{equation*}
\end{lemma}

\begin{proof}
	For each prime $P$, 	
	$E\left[W_P^m\right]=0$ if $m$ is odd and  $E\left[W_P^m\right]=1$ if $m$ is even. Therefore, its follows from Lebesgue dominated convergence theorem and the formula of $d_k(n)$ in \cite{Tit}  that 	

	\begin{equation}\label{e=dk}
	\begin{split}
	E\left[L_P(1,.)^k\right] & = E\left[\left(\sum_{n=0}^\infty  \frac{W_P^n}{|P|^n}\right)^k\right]\\
	& = E\left[\sum_{m=0}^\infty \frac{W_P^m}{|P|^m} \sum_{m = n_1+n_2+\cdots+n_k}1\right]\\
	& = E\left[\sum_{m=0}^\infty \frac{W_P^m}{|P|^m} \frac{(k+m-1)!}{m!(k-1)!}\right]\\
	& = \sum_{m=0}^\infty \frac{(k+m-1)!}{m!(k-1)!|P|^m} E\left[W_P^m\right]\\
	& = \sum_{n=0}^\infty \frac{d_k(P^{2n})}{|P|^{2n}}.
	\end{split}	
	\end{equation}
	From this and the fact that $d_k(P)\ll_{k,\varepsilon}|P|^\varepsilon$  we establish 	
	
	\begin{equation*}
	\begin{split}
	E\left[L_P(1,.)^k\right] 
	& = 1 + O_k\left(|P|^{2 \varepsilon-2}\right).
	\end{split}
	\end{equation*}
	Since $\sum_P |P|^{2 \epsilon-2}<\infty,$  the infinite product $\prod_P E\left[L_P(1,.)^k \right]$ is convergent.\\
	Now, for $n\ge 2$ put $
	Y_n(\omega):= \prod_{\substack{\deg(P)\le n}} L_P(1,\w)^k,$  and $Y(\w)=L(1,\w)^k.$ 
	Since $W_P$'s are independent random variables, then 
	
	\begin{equation}\label{Lpn}
	\begin{split}
	\prod_{\substack{P \text{ irreducible} \\ \deg(P)\le n}} E\left[L_P(1,.)^k\right] 
	& = E\left[Y_n\right].
	\end{split}
	\end{equation}
	Moreover, using (\ref{e=dk}), (\ref{dkp}) and the independence of $W_P$'s, we have
	
	\begin{equation*}
	\begin{split}
	E\left[|Y_n|^2\right] 
	&= \prod_{\substack{P \text{ irreducible} \\ \deg(P)\le n}} E\left[ L_P(1,.)^{2k} \right]\\
	&= \prod_{\substack{P \text{ irreducible} \\ \deg(P)\le n}} E\left[  1 +  \sum_{n=1}^\infty \frac{d_{2k}(P^{2n})}{|P|^{2n}} \right]\\
	&= \prod_{\substack{P \text{ irreducible} \\ \deg(P)\le n}} \left(1+ O_k\left(|P|^{2\varepsilon-2}\right)\right) <C_k,
	\end{split}
	\end{equation*}
	where $C_k>0$ is constant depend on $k$. Making use of Lemma 3 from \cite{Shiryaev}, we get that the sequence $\{Y_n\}$ is uniformly integrable. Recall that $	\prod_{\substack{ \deg(P)\le n}}L_P(1,\w)\to L(1,\w)$ as $n\to\infty$ for almost all $\w\in\Omega.$ Therefore, $Y_n(\omega)\to Y(\omega)$ as $n\to\infty$ for almost all $\w\in\Omega.$ Since $\{Y_N\}$ is uniformly integrable and by Theorem 4(b) from \cite{Shiryaev}, we have that $Y$ ia also integrable and $E\left[Y_N\right]\to E\left[Y\right]$ as $n\to\infty$. Combining this with (\ref{Lpn}) we complete the proof.
\end{proof}

\begin{lemma}\label{E=dk}
	For $k>0$, we have that
	
	\begin{equation*}
	E\left[L(1,.)^k\right]= \sum_{f \text{ monic}} \frac{d_k(f^2)}{|f|^2}.
	\end{equation*}
\end{lemma}

\begin{proof}
	Recall that $d_k(n)$ is a multiplicative function (see p.5 in \cite{Tit}), and from Lemma \ref{E=EP}, and (\ref{e=dk}) we can write 

	\begin{equation*}
	\begin{split}
	\sum_{f \text{ monic}} \frac{d_k(f^2)}{|f|^2} 
	&= \prod_{\substack{P \text{ monic} \\ \text{irreducible}}} \sum_{n=0}^\infty \frac{d_k(P^{2n})}{|P|^{2n}}\\
	&= E\left[L(1,.)^k\right].
	\end{split}
	\end{equation*}
\end{proof}

\begin{lemma} Let $f\in A$ be monic polynomial, we have that

	\begin{equation*}
	E\left[W_f\right]= \begin{cases}
	0 & \text{ if } f \text{ is not a square}\\
	1  & \text{ if } f \text{ is a square}
	\end{cases}
	\end{equation*}
\end{lemma}

\begin{proof}
	Let $f=P_1^{e_1}\cdots P_r^{e_r}$ be the prime power factorization of $f$. By the independence of $W_P$'s we have 
	\begin{equation*}
	\begin{split}
	E\left[W_f\right] &= E\left[W_{P_1}^{e_1}\right]\cdots E\left[W_{P_r}^{e_r}\right]\\
	&= \prod_{i=1}^r E\left[W_{P_i}^{e_i}\right].
	\end{split}
	\end{equation*}
	Since $E\left[W_P^{e_i}\right]=0$ when $e_i$ is odd and $E\left[W_P^{e_i}\right]=1$  when $e_i$ is even, we obtain the Lemma.
\end{proof}


\section{Nagoshi's Theorems in Function Fields}

\subsection{Moments of $L(1,\x_P)$}\label{Proof of Nagoshis theorem} $\color{white}gkh$\\

For $k\in\N$, we define 
\begin{equation}\label{ak}
a_k := \sum_{f \text{ monic}} \frac{d_k(f^2)}{|f|^2}\in \R,
\end{equation}
which is convergent by the bound $d_k(f)\ll_{k,\varepsilon}|f|^\varepsilon,$ for any $\varepsilon>0$ (see Theorem 2 in \cite{kara+voron}). Remember that the set $\mathbb{P}_n$ is defined by

\begin{equation*}
\mathbb{P}_n=\left\{P\in \f : P \text{ monic, irreducible,} \deg(P)= n \right\}.
\end{equation*}
We now state the main results of this section which can be seen as the function field analogue of Theorem \ref{nagoshithm1}. 

\begin{theorem}\label{nagoshis theorem}\label{L(1.xp)^k}
	Let $k\in\N$, $q$ be a fixed power of an odd prime, we have that
	
	\begin{equation*}
	\sum_{P\in\p} \log|P| L(1,\x_P)^k = |P| a_k + O\left(|P|^{\frac{1}{2}} \left(\log|P|\right)^{k+1}\right),
	\end{equation*}
	where $a_k$ defined as in (\ref{ak}).
\end{theorem}

Before we prove the main result we need the following two Lemmas.

\begin{lemma}\label{bound}
	For $k\in\N , k\ge 2,$ and $f$ monic polynomial in $\f$. We have that
	
	\begin{equation*}
	\begin{split}
	1.\text{ }\; q^{-kg} \log|P| \sum_{P\in\p} &\sum_{\substack{f \text{ monic} \\ \deg(f)\le kg-1}} d_k(f) \x_P(f)\\
	& \ll |P|^{1-\frac{k}{2}} q^{\left[\frac{kg-1}{2}\right]} \left(\log|P|\right)^{\frac{k(k+1)}{2}-1}.
	\end{split}
	\end{equation*}
	
	\begin{equation*}
	2.\text{ }\; \log|P| \sum_{P\in\p} \sum_{\substack{f \text{ monic} \\ \deg(f)\le kg \\ f\neq\Box}} \frac{d_k(f)}{|f|} \x_P(f) \ll |P|^{\frac{1}{2}} \left(\log|P|\right)^{k+1}.
	\end{equation*} 
\end{lemma}

\begin{proof} $\color{white} rmm$\\
	Put

	\begin{equation*}
	\begin{split}
	I_1  =  q^{-kg} \log|P| \sum_{\substack{f \text{ monic} \\ \deg(f)\le kg-1 \\ f=\Box}} d_k(f) \sum_{P\in\p} \x_P(f)
	\end{split}
	\end{equation*}
	and 

	\begin{equation*}
	\begin{split}
	I_2 = q^{-kg} \log|P| \sum_{\substack{f \text{ monic} \\ \deg(f)\le kg-1 \\ f\neq\Box}} d_k(f) \sum_{P\in\p} \x_P(f),
	\end{split}
	\end{equation*}
	then we have

	\begin{equation*}
	q^{-kg} \log|P| \sum_{P\in\p} \sum_{\substack{f \text{ monic} \\ \deg(f)\le kg-1}} d_k(f) \x_P(f) = I_1 + I_2,
	\end{equation*}
Consider the sum $I_{1},$ since $f$ is a perfect square then we can write $f=l^2, l\in A$. Making use of the Prime Polynomial Theorem \ref{PNT} and Lemma \ref{Dlk}  

\begin{equation*}
\begin{split}
I_{1} 
& \ll |P|^{1-\frac{k}{2}} q^{\left[\frac{kg-1}{2}\right]} \left(\log|P|\right)^{\frac{k(k+1)}{2}-1}.
\end{split}
\end{equation*}
Applying Lemma \ref{nkbox}, $I_2$ is bounded by 

\begin{equation*}
I_{2}\ll |P|^{\frac{1}{2}} \left(\log|P|\right)^k.
\end{equation*}
Hence we obtain the first part of the Lemma. For the second part it follows from Lemma \ref{nkbox}.
\end{proof}

The next lemma we need is the following.

\begin{lemma}\label{main1} For $k\in\N , k\ge 2,$ and $f$ monic polynomial in $\f$. We have
	
	\begin{equation*}
	\begin{split}
	\log|P| \sum_{P\in\p} \sum_{\substack{f \text{ monic} \\ \deg(f)\le kg \\ f=\Box}} \frac{d_k(f)}{|f|} \x_P(f) = |P| a_k + O\left(|P|  q^{-\left[\frac{kg}{2}\right]}  \left(\log|P|\right)^{\frac{k(k+1)}{2}-1} \right),
	\end{split}
	\end{equation*}  
	where $a_k$ is defined as in (\ref{ak}).
\end{lemma}

\begin{proof}
Write $f=l^2, l\in A$, Since $f$ is a perfect square, then we have $\x_P(l^2)=1$ for $(P,l)=1$ and $\deg(l)<$ $\deg(P)=2g+1$. By the Prime Polynomial Theorem \ref{PNT} we have  

\begin{equation*}\label{NoI11}
\begin{split}
\log|P| \sum_{P\in\p} \sum_{\substack{f \text{ monic} \\ \deg(f)\le kg \\ f=\Box}} &\frac{d_k(f)}{|f|} \x_P(f) \\
=&  |P| \sum_{\substack{l \text{ monic} \\ \deg(l)\le \left[\frac{kg}{2}\right] }} \frac{d_k(l^2)}{|l|^2} +O\left(|P|^{\frac{1}{2}} \sum_{\substack{l \text{ monic} \\ \deg(l)\le \left[\frac{kg}{2}\right] }} \frac{d_k(l^2)}{|l|^2}\right).
\end{split}
\end{equation*}  
From Lemma \ref{dl} the O-term is bounded by $|P|^{\frac{1}{2}} $. 
%
%
For the main term we have

\begin{equation*}
\begin{split}
|P| \sum_{\substack{l \text{ monic} \\ \deg(l)\le \left[\frac{kg}{2}\right] }} \frac{d_k(l^2)}{|l|^2}&= |P| \left(\sum_{\substack{l \text{ monic} }} \frac{d_k(l^2)}{|l|^2} - \sum_{\substack{l \text{ monic} \\ \deg(l)> \left[\frac{kg}{2}\right] }} \frac{d_k(l^2)}{|l|^2}\right)\\
&= |P| a_k + O\left(|P|  q^{-\left[\frac{kg}{2}\right]}  \left(\log|P|\right)^{\frac{k(k+1)}{2}-1}\right),
\end{split}
\end{equation*}  
where $a_k$ is defined in (\ref{ak}).
\end{proof}

We are now in a position to prove the main result of this section.

\begin{proof}[Proof of Theorem \ref{nagoshis theorem}]
 
From the ``approximate" functional equation (\ref{aproxL(1,xp)k}) we have 

\begin{equation}\label{ap1}
\begin{split}
\sum_{P\in\p} \log|P| & L(1,\x_P)^k \\
= & \log|P| \sum_{P\in\p} \Bigg\{\sum_{\substack{f \text{ monic} \\ \deg(f)\le kg \\ f=\Box}} \frac{d_k(f)}{|f|} \x_P(f) + \sum_{\substack{f \text{ monic} \\ \deg(f)\le kg \\ f\neq\Box}} \frac{d_k(f)}{|f|} \x_P(f) \\
&\text{\color{white} l}  + q^{-kg}  \sum_{\substack{f \text{ monic} \\ \deg(f)\le kg-1}} d_k(f)  \x_P(f) \Bigg\}.
\end{split}
\end{equation}
Applying Lemma \ref{bound} and Lemma \ref{main1} in (\ref{ap1}) we obtain the Theorem \ref{L(1.xp)^k}.

\end{proof}


\subsection{Extending Nagoshi's Results} \label{Proof of L(1.xp)^k} $\color{white}gbg$\\

In this section we extend Theorem \ref{L(1.xp)^k} and write the sum $a_k$ in to a specific form that is more suitable for the calculations that we present in this section. We start with the following lemma.

\begin{lemma}\label{main2} For $k\in\N , k\ge 2,$ and $f$ monic polynomial in $\f$. We have that

	\begin{equation*}
	\begin{split}
	\log|P| \sum_{P\in\p} \sum_{\substack{f \text{ monic} \\ \deg(f)\le kg \\ f=\Box}} & \frac{d_k(f)}{|f|} \x_P(f) \\
	& = |P| B_k +O\left(|P|^{\frac{1}{2}} q^{-\left[\frac{kg}{2}\right]} \left(\lo|P|\right)^{\frac{k(k+1)}{2}-1}\right),
	\end{split}
	\end{equation*}  
	where 	

	\begin{equation*}
	\begin{split}
	B_k= &  \sum_{n=0}^{\left[kg/2\right]} \left(\sum_{i=1}^{k(k+1)/2} c_{-i} \binom{n+i-1}{i-1} \left(-q\right)^i\right) q^{-n}.
	\end{split}
	\end{equation*}
	The sum is parenthesis is a polynomial in $n$ of degree $\frac{k(k+1)}{2}-1$ with leading term 

	\begin{equation*}
	\frac{A_k(1)}{\left(\frac{k(k+1)}{2}-1\right)!} n^{\frac{k(k+1)}{2}-1}.
	\end{equation*}
\end{lemma}

\begin{proof}
As in Lemma \ref{main1}, write $f=l^2, l\in A$, then from the Prime Polynomial Theorem \ref{PNT} we have 

\begin{equation*}\label{oI11}
\begin{split}
I= \log|P| & \sum_{P\in\p} \sum_{\substack{f \text{ monic} \\ \deg(f)\le kg \\ f=\Box}} \frac{d_k(f)}{|f|} \x_P(f) \\
&  = |P| \sum_{n=0}^{\left[kg/2\right]} q^{-2n} \sum_{\substack{l \text{ monic} \\ \deg(l)=n}} d_k(l^2) +O\left(|P|^{\frac{1}{2}} \sum_{n=0}^{\left[kg/2\right]} q^{-2n} \sum_{\substack{l \text{ monic} \\ \deg(l)=n}} d_k(l^2) \right).
\end{split}
\end{equation*}  
Using Lemma \ref{Dlk} we have

\begin{equation*}
\begin{split}
I   &= |P| \sum_{n=0}^{\left[kg/2\right]} q^{-2n} q^n \left(\sum_{i=1}^{(k(k+1))/2} c_{-i} \binom{n+i-1}{i-1} \left(-q\right)^i\right)\\
& \text{\color{white}nbncjnbn mbn}+O\left(|P| \sum_{n=0}^{\left[kg/2\right]} q^{-2n} q^{\epsilon n} \right) +O\left(|P|^{\frac{1}{2}} \sum_{n=0}^{\left[kg/2\right]} q^{-2n} q^n n^{\frac{k(k+1)}{2}-1} \right)\\
& \\
& = |P| \sum_{n=0}^{\left[kg/2\right]} \sum_{i=1}^{(k(k+1))/2} c_{-i} \binom{n+i-1}{i-1} \left(-q\right)^i q^{-n} +O\left(|P|  q^{(\epsilon-2)\left[\frac{kg}{2}\right]}  \right) \\
& \text{\color{white}nbncjnbn mbn} +O\left(|P|^{\frac{1}{2}} q^{-\left[\frac{kg}{2}\right]} g^{\frac{k(k+1)}{2}-1} \right)\\
&\\
& = |P| B_k +O\left(|P|^{\frac{1}{2}} q^{-\left[\frac{kg}{2}\right]} \left(\lo|P|\right)^{\frac{k(k+1)}{2}-1}\right).\\
\end{split}
\end{equation*}  
\end{proof}

From Lemma \ref{bound}, Lemma \ref{main2} and equation (\ref{ap1}) we establish the following theorem.

\begin{theorem}\label{L(1.xp)kk}
Let $k\in\N$, $q$ be a fixed power of an odd prime. We have that
	
\begin{equation*}
\begin{split}
&\sum_{P\in\p} \log|P| L(1,\x_P)^k = |P| B_k + O\left(|P|^{\frac{1}{2}}  \left(\log|P|\right)^{k+1}\right).
\end{split}
\end{equation*}
where $B_k$ is defined as in Lemma \ref{main2}.
\end{theorem}

For any non-constant irreducible polynomial $P\in A$ with sgn$(P)\in\{1,\gamma\},$ where $\gamma$ is a fix generator of $\mathbb{F}_q^\times$, let $\mathcal{O}$ be the integer closure of $A$ in the quadratic function field $k\left(\sqrt{P}\right)$. Let $h_P$ be the ideal class number of $\mathcal{O}$, and $R_P$ be the regulator of $\mathcal{O}$ if $\deg(P)$ is even and sgn$(P)=1$. We have a formula, quoted from \cite{Rosen} Theorem 17.8, which connects $L(1,\x_P)$ with $h_P$, namely

\begin{equation}\label{hP}
L(1,\x_P) = \begin{cases}
\sqrt{q} |P|^{-\frac{1}{2}} h_P  & \text{if } \deg(P) \text{ is odd,}\\
(q-1) |P|^{-\frac{1}{2}} h_P R_P  & \text{if } \deg(P) \text{ is even and sgn}(P)=1,\\ 
\frac{1}{2} (q+1) |P|^{-\frac{1}{2}} h_P  & \text{if } \deg(P) \text{ is even and sgn}(P)=\gamma.
\end{cases}
\end{equation}

Combining Theorem \ref{L(1.xp)kk} and equation (\ref{hP}), we obtain the following corollary.

\begin{corollary}
	Let $q$ be a fixed power of an odd prime. Then with the same notation as in Lemma \ref{main2}, we have that
	
	\begin{equation*}
	\begin{split}
	&\sum_{P\in \p} \left(h_P\right)^k = \frac{|P|^{1+\frac{k}{2}}}{\log|P|} q^{-\frac{k}{2}} B_k + O\left(|P|^{\frac{k+1}{2}}  \left(\log|P|\right)^{k}\right).
	\end{split}
	\end{equation*}
\end{corollary}


\subsection{The Second Moment of $L(1,\x_P)$} \label{Proof of L(1,xp)^2} $\color{white}gbg$\\

We start this section proving the following lemma.

\begin{lemma}\label{main21}
	Let $f$ monic polynomial in $A=\f$. We have
	
	\begin{equation*}
	\begin{split}
	\log|P| &\sum_{P\in\p} \sum_{\substack{f \text{ monic} \\ \deg(f)\le 2g \\ f=\Box}}  \frac{d(f)}{|f|} \x_P(f) \\
	 = &|P| \frac{1}{2}\z(2)^2 q^{-2} \Bigg(q^{-g-1} ( g^2 \left(-q^2+2 q-1\right)+g \left(-5 q^2+4 q+1\right)\\
	& -6 q^2 )+ 2 q^2+2 q+2 \Bigg)  +O\left(\left(\log|P|\right)^{2} \right).
	\end{split}
	\end{equation*}  
\end{lemma}

\begin{proof}
	Write $f=l^2$, since $\x_P(l^2)=1$ for $(P,l)=1$ and $\deg(l)< \deg(P)=2g+1$, using the Prime Polynomial Theorem \ref{PNT} we have 
	
	\begin{equation*}\label{oT11}
	\begin{split}
	T= \log|P| &\sum_{P\in\p} \sum_{\substack{f \text{ monic} \\ \deg(f)\le 2g \\ f=\Box}}  \frac{d(f)}{|f|} \x_P(f) \\
	&  = |P| \sum_{n=0}^{g} q^{-2n} \sum_{\substack{l \text{ monic} \\ \deg(l)=n}} d(l^2) +O\left(|P|^{\frac{1}{2}} \sum_{n=0}^{g} q^{-2n} \sum_{\substack{l \text{ monic} \\ \deg(l)=n}} d(l^2) \right).
	\end{split}
	\end{equation*}  
	Using Lemma \ref{Dlk} we have
	
	\begin{equation*}
	\begin{split}
	T  
	 =& |P| \sum_{n=0}^{g} q^{-n} \left(1+\frac{1}{2}\left(3+q^{-1}\right)n+\frac{1}{2}\left(1-q^{-1}\right)n^2\right)  +O\left(|P|^{\frac{1}{2}} \sum_{n=0}^{g} q^{-n} n^{2} \right)\\
	= &|P| \frac{1}{2}\z(2)^2 q^{-2} \Bigg(q^{-g-1} ( g^2 \left(-q^2+2 q-1\right)+g \left(-5 q^2+4 q+1\right)\\
	&\ \ \  -6 q^2 )+ 2 q^2+2 q+2 \Bigg)  +O\left(\left(\log|P|\right)^{2} \right).
	\end{split}
	\end{equation*}  	
This proves the lemma.
\end{proof}

Now, consider the ``approximate" functional equation (\ref{aproxL(1,xp)k}) when $k=2$, 

\begin{equation*}
\begin{split}
\sum_{P\in\p} & \log|P| L(1,\x_P)^2 \\
 &=\log|P| \sum_{P\in\p} \Bigg\{\sum_{\substack{f \text{ monic} \\ \deg(f)\le 2g \\ f=\Box}}  \frac{d(f)}{|f|} \x_P(f) + \sum_{\substack{f \text{ monic} \\ \deg(f)\le 2g \\ f\neq\Box}}  \frac{d(f)}{|f|} \x_P(f)\\ 
&\ \ \  + q^{-2g} \sum_{\substack{f \text{ monic} \\ \deg(f)\le 2g-1}} d(f)  \x_P(f)\Bigg\}.\\
\end{split}
\end{equation*}	

From Lemma \ref{main21} and Lemma \ref{bound} with $k=2$ we proved the asymptotic formula for the second moment of $L(1,\x_P)$. 

\begin{theorem}\label{L(1,xP)^2}
	\begin{equation*}
	\begin{split}
	&\sum_{P\in\p} \log|P| L\left(1,\x_P\right)^2 =  |P| \z(2)^2 q^{-2} \left( q^2+ q+ 1\right) + O\left(|P|^{\frac{1}{2}}  \left(\log|P|\right)^{3}\right).
	\end{split}
	\end{equation*}
\end{theorem}


\subsection{Applying Theorem \ref{L(1.xp)^k} when $k=2$}$\text{\color{white}dlkjh}$\\

In this section we use Theorem \ref{L(1.xp)kk} to obtain an explicit formulae for the second moment of quadratic Dirichlet $L$-functions associated to $\x_P$ over function fields, then compare it with the result that we established in Section \ref{Proof of L(1,xp)^2}. For $k=2$, 

\begin{equation*}
\begin{split}
\sum_{i=1}^{3}& c_{-i} \binom{n+i-1}{i-1} \left(-q\right)^i \\
& =-\frac{1}{2} c_{-3}  q^3 n^2+ \left(c_{-2} q^2-\frac{3}{2} c_{-3} q^3\right) n -\left(c_{-3} q^3 - c_{-2} q^2+ c_{-1} q\right),
\end{split}
\end{equation*}
and so 

\begin{equation*}
\begin{split}
B_2
=& \sum_{n=0}^{g} \left(-\frac{1}{2} c_{-3}  q^3 n^2+ \left(c_{-2} q^2-\frac{3}{2} c_{-3} q^3\right) n -\left(c_{-3} q^3 - c_{-2} q^2+ c_{-1} q\right)\right) q^{-n}\\
=&  \z(2)^2 q^{-2} \left(-\frac{ q^6}{q-1} c_{-3} +c_{-2} q^4-c_{-1} (q-1) q^2 \right)+O\left(|P|^{-\frac{1}{2}} \left(\log|P|\right)^2\right).
\end{split}
\end{equation*}

Hence, the second moment using Theorem \ref{L(1.xp)kk} is

\begin{equation}\label{nagoshk=2}
\begin{split}
& \sum_{P\in\p} \log|P|L\left(1,\x_P\right)^2 \\
&= |P|B_2 + O\left(|P|^{\frac{1}{2}} \left(\log|P|\right)^3\right)\\
&= \z(2)^2 |P| q^{-2} \left(-\frac{ q^6}{q-1} c_{-3} +c_{-2} q^4-c_{-1} (q-1) q^2 \right)+ O\left(|P|^{\frac{1}{2}}  \left(\log|P|\right)^{3}\right).
\end{split}
\end{equation}

We know that $c_{-i}, i=1,2,3,$ are actually constants and with simple arithmetic calculation we can see that Theorem \ref{L(1,xP)^2} agrees with equation (\ref{nagoshk=2}) when we have the following relation 

\begin{equation*}
\begin{split}
c_{-1} &=  \frac{q^2+q+1}{(1-q) q^2} -\frac{q^2}{1-q} c_{-2} -\frac{q^4}{(1-q)^2} c_{-3}.\\
\end{split}
\end{equation*}


\subsection{The Statistic Distribution of Class Number}${\color{white}kljfhkfj}$\\

Define $F$ to be the distribution function of the random variable $L(1,\w), \w\in\Omega$, by

\begin{equation}\label{F(x)}
F(x) := \pp\left(\left\{\w\in\Omega \mid L(1,\w)\le x\right\}\right) \text{ for }x\in\mathbb{R}.
\end{equation}
Moreover, let $\widetilde{F}$ be the distribution function of the random variable $\ln L(1,\w), \w\in\Omega$, that is

\begin{equation}
\widetilde{F}(x) := \pp\left(\left\{\w\in\Omega \mid \ln L(1,\w)\le x\right\}\right) \text{ for }x\in\mathbb{R}.
\end{equation}
Its clear that $\widetilde{F}(x)=F(e^x)$ for $x\in\R$. Note that the $E\left[L(1,.)^{it}\right], t\in\R,$ is the characteristic function of $\widetilde{F}(x).$ 
\newline

In this section we give the proof of the function field analogue of Theorem \ref{nagoshiDistrb}. Our main result is:

\begin{theorem}\label{theorem3}
	Let $x\in\R$, for $n$ odd we have,
	
		\begin{equation} \label{1.7}
		\lim_{n\to \infty} \frac{1}{\left|\mathbb{P}_n\right|}\left|\left\{  P\in\mathbb{P}_n\mid h_P\le q^{-\frac{1}{2}} |P|^{\frac{1}{2}} e^x
		\right\}\right|=F(e^x)=\widetilde{F}(x),
		\end{equation}	
		and for $n$ even we have
		
		\begin{equation}\label{1.8}
		\lim_{n\to \infty} \frac{1}{\left|\mathbb{P}_n\right|}\left|\left\{P\in\mathbb{P}_n\mid  h_pR_P \le \left(q-1\right)^{-1}|P|^{\frac{1}{2}} e^x \right\}\right|=F(e^x)=\widetilde{F}(x).
		\end{equation}	
	
	Moreover the characteristic function of $F$ has the form
	
	\begin{equation}\label{powerseries}
	1+\sum_{k=1}^\infty \frac{a_k}{k!}\left(it\right)^k, \;\; t\in \R,
	\end{equation}
	where the numbers $a_k$ are as in equation (\ref{ak}). The characteristic function $E\left[L(1,.)^{it}\right]$ of $\widetilde{F}$ has the form
	
	\begin{equation}\label{prod}
	\prod_{\substack{P \text{ monic} \\ \text{irreducible}}} \left(\frac{1}{2} \left(1-\frac{1}{|P|}\right)^{-it} + \frac{1}{2}\left(1+\frac{1}{|P|}\right)^{-it}\right); \;\; t\in\R,
	\end{equation}
	and satisfies
	
	\begin{equation*}
	E\left[L(1,.)^{it}\right] \ll \exp\left(-c \frac{|t|}{\ln\left(2+|t|\right)}\right) \;\; \text{ for all } t\in\R
	\end{equation*}
	with absolute constant $c>0.$ The distribution function $\widetilde{F}$ has a density $f$. Further, $\widetilde{F}$ and $f$ are infinitely differentiable.
\end{theorem}

We first prove the following auxiliary lemma.

{
\begin{lemma}\label{Barban lemma 5.8}
We have the following estimate, 

\begin{equation*}
\sum_{f \text{ monic}} \frac{d_k(f^2)}{|f|^2} \ll c^{k\ln\ln k},
\end{equation*}
where $c$ is an absolute constant depend on $q$.
\end{lemma}

\begin{proof}
	Clearly,
	
	\begin{equation*}
	1+\frac{d_k(P^2)}{|P|^2} +\frac{d_k(P^4)}{|P|^4} + \cdots < 1+ \frac{k^2}{|P|^2} \frac{1}{\left(1-|P|^{-1}\right)^k}.
	\end{equation*}
	So we have that,
	
	\begin{equation*}
	\begin{split}
	\prod_{\substack{P \text{ prime} \\ |P|\le k}} \left(1+ \frac{k^2}{|P|^2} \frac{1}{\left(1-|P|^{-1}\right)^k}\right) &< k^{c\pi(\log k)} \left( \prod_{\substack{P \text{ prime} \\ |P|\le k}} \frac{1}{1-|P|^{-1}} \right)^k\\
	& < c^{k\ln\ln k}
	\end{split}
	\end{equation*}
	and 
	
		\begin{equation*}
	\begin{split}
	\prod_{\substack{P \text{ prime} \\ |P|> k}} \left(1+ \frac{k^2}{|P|^2} \frac{1}{\left(1-|P|^{-1}\right)^k}\right) &<  \prod_{\substack{P \text{ prime} \\ |P|> k}} \left(1+ \frac{ck^2}{1-|P|^{-2}} \right)< c^k.
	\end{split}
	\end{equation*}
	By using the Euler product we complete the proof.
\end{proof}}

We are now in a position to present the proof of the main result in this section.

\begin{proof}[Proof of Theorem \ref{theorem3}]
Now, Theorem \ref{nagoshis theorem} and simple arithmetic manipulation yield that, for $k\in\mathbb{N},$

\begin{equation}\label{1}
\sum_{P\in \mathbb{P}_n} L\left(1,\x_P\right)^k \sim \frac{a_k|P|}{\lo|P|} \text{  as } n\to\infty,
\end{equation}
since 

\begin{equation}\label{2}
 a_k \ll e^{c_1 k \log\log k},
\end{equation}
(see Lemma \ref{Barban lemma 5.8} above), where we can see that the power series $1+\sum_{k=1}^\infty a_kw^k/k!$ has infinite radius of convergence. 
From Lemma \ref{E=dk}, Theorem 30.1 in \cite{Billingsley} and Lemma 5.7 in \cite{Barban}, we can deduce that $F$ defined in (\ref{F(x)}) is the unique distribution function with the moments $a_1,a_2,...$. Therefore, since $\#\mathbb{P}_n\sim|P|/\log|P|$ as $n\to\infty$, it follows from the method of moments, Theorem 30.2 in \cite{Billingsley} and (\ref{1}) that 

\begin{equation}\label{5.3}
\lim_{n\to\infty}\frac{1}{\left|\mathbb{P}_n\right|} \left|\left\{P\in\mathbb{P}_n \mid L(1,\x_P)\le y\right\}\right|= F(y)
\end{equation}
for each $y\in \R$ at which $F$ is continuous. Moreover, we obtain equation (\ref{powerseries}) from (\ref{2}) and [\cite{Billingsley},(26.7)] or [\cite{Barban},Lemma 5.7]. 

Now, for any fixed $t\in\R$, by the independence of $W_P$'s (see (26.12) in \cite{Billingsley}) we have

\begin{equation*}
\prod_{\substack{P \text{ irreducible} \\ \deg P\le n}} E\left[L_P(1,.)^{\i t}\right]= E\left[\prod_{\substack{P \text{ irreducible} \\ \deg P\le n}} L_P(1,.)^{\i t}\right].
\end{equation*}
Note that 

\begin{equation*}
\prod_{\substack{\deg P\le n}} E\left[L_P(1,\w)^{\i t}\right]\to E\left[L(1,\w)^{\i t}\right],
\end{equation*}
as $n\to\infty$ for almost all $\w\in\Omega$, and that  

\begin{equation*}
\begin{split}
\prod_{ \deg P\le n}& E\left[L_P(1,\w)^{\i t}\right]\\ 
& = \cos\left( t\log\prod_{ \deg P\le n} L_P(1,\w)\right)+ \i \sin\left( t\log\prod_{ \deg P\le n} L_P(1,\w)\right),
\end{split}
\end{equation*}
where $\cos(\cdot)$ and $\sin(\cdot)$ above are bounded uniformly for all $\w$ and $n$, therefore we deduce from Lebesgue's dominated convergent Theorem that 

\begin{equation}\label{4}
\prod_{\substack{P \text{ irreducible} \\ \deg P\le n}} E\left[L_P(1,\cdot)^{\i t}\right] \to E\left[ L(1,\cdot)^{\i t}\right] 
\end{equation}
as $n\to \infty$ for any fixed $t\in\R$. Recall the Taylor's series $\left(1+x\right)^k=\sum_{m=0}^\infty \binom{\alpha}{m}x^m,$ since we have for $|t|/|P|$ small

\begin{equation}\label{5}
\begin{split}
E\left[L_P(1,\cdot)^{\i t}\right] &= \frac{1}{2} \left(1-\frac{1}{|P|}\right)^{-\i t} + \frac{1}{2} \left(1+\frac{1}{|P|}\right)^{-\i t}\\
& = 1-\frac{t^2-\i t}{2 |P|^2} +O\left( \frac{|t|^3+ |t|^2+|t|}{|P|^3}\right)
\end{split}
\end{equation} 
where the infinite product $\prod_{P} E\left[L_P(1,\cdot)^{\i t}\right]$ converges absolutely and uniformly for $t$ in any compact subset of $\R.$ Hence, making use of (\ref{4}), we obtain (\ref{prod}) and that $E\left[L(1,\cdot)^{\i t}\right]$ is a continuous function on $\R$. 

Let $c_q\ge q>4$ be a positive constant depending on $q$. If $|t|>c_1$ and $|P|\ge c_q|t|$, then we obtain from (\ref{5}) that 

\begin{equation*}
\left|E\left[L_P(1,\cdot)^{\i t}\right]\right| \le 1-\frac{t^2}{2|P|^2}.
\end{equation*} 
Since $\left|E\left[L_P(1,\cdot)^{\i t}\right]\right| \le 1,$ we have that for any real numbers $q\le y_1<y_2$

\begin{equation*}
E\left[L(1,\cdot)^{\i t}\right] \le \prod_{\substack{y_1\le |P| \le y_2}} E\left[L_P(1,\cdot)^{\i t}\right].
\end{equation*}
By choosing $y_1=c_q|t|$ and $y_2=2c_q|t|$ we have for any $t\in\R$ with $|t|$ large

\begin{equation*}
\begin{split}
\left|E\left[L(1,\cdot)^{\i t}\right]\right| & \le \prod_{\substack{\lo c_q|t|\le \deg(P) \le \lo 2c_q|t|}} E\left[L_P(1,\cdot)^{\i t}\right]\\
& \le \exp\left(-t^2\sum_{r=\log c_q|t|}^{ \lo 2c_q|t|} \frac{1}{r q^r} \right)\\
& \le \exp\left(-\widetilde{c_{q}}\frac{|t|}{ \log|t| }\right). 
\end{split}
\end{equation*}   

Then, by the continuity of $E\left[L\left(1,\cdot\right)^{\i t}\right]$ we have

\begin{equation}\label{E bound}
E\left[L\left(1,\cdot\right)^{\i t}\right] \ll \exp\left(-\widetilde{c_{q}}\frac{|t|}{ \log\left(2+|t|\right) }\right)
\end{equation}
for all $t\in\R,$ which gives $\int_{-\infty}^{\infty} \left| E\left[L\left(1,\cdot\right)^{\i t}\right] \right|  <0$. Therefore, using the inversion formula (Theorem 26.2 in \cite{Billingsley}) we have that $\widetilde{F}$ has a density $f$. Moreover, using similar reasoning as presented in \cite[pg.344-347]{Billingsley} and by making use of equation (\ref{E bound}) we can concluded that the density $f$ and the function $\widetilde{F}$ is differentiable on $\R$. In particular, the function $F$ is continuous on $\left(0,\infty\right).$ Hence, from the above with (\ref{5.3}) and Dirichlet's class number formula we have equations (\ref{1.7}) and (\ref{1.8}).

\end{proof}


\section{Complex Moments of $L(1,\x_P)$}\label{Proof of lumley theorem}

In this part we investigate the complex moments of $L(1,\x_P)$, where $\x_P$ varies over quadratic characters associated to irreducible polynomials $P$ of degree $n$ over $\mathbb{F}_{q}$, in a large uniform range. We express the complex moments of $L(1,\x_P)$ as follows. 

\begin{theorem}\label{lumley theorem}
	Let $n$ be positive integer, and let $z\in\C$ such that $|z|\le \frac{\log |P|}{260 \log_2|P|\ln\log_2|P|}$. Then	
	
	\begin{equation*}
	\frac{1}{\left|\mathbb{P}_n\right|} \sum_{P\in\mathbb{P}_n}L(1,\x_P)^z =  \sum_{f \text{ monic}} \frac{d_z(f^2)}{|f|^2} \left(1+O\left(\frac{1}{\left(\log|P|\right)^{11}}\right)\right).
	\end{equation*}
\end{theorem}

An applications of the above Theorem and Artin's class number formula over function fields (\ref{hP}) we obtain some corollaries for the average size of the class number $h_P$ over $\mathbb{P}_n$ when we specialize $n$ to be $n=2g+1$ and $n=2g+2$ and letting the genus $g\to \infty$.

\begin{corollary}\label{lumley coral1}
	Let $z\in\C$ such that $|z|\le \frac{g}{130 \log(g)\ln\log(g)}$. Then	
	
	\begin{equation*}
	\frac{1}{\left|\p\right|} \sum_{P\in\p} h_P^z = q^{gz} \sum_{f \text{ monic}} \frac{d_z(f^2)}{|f|^2} \left(1+O\left(\frac{1}{g^{11}}\right)\right).
	\end{equation*}
\end{corollary}

\begin{corollary}\label{lumley coral2}
	Let $z\in\C$ such that $|z|\le \frac{g}{130 \log(g)\ln\log(g)}$. Then	
	
	\begin{equation*}
	\frac{1}{\left|\p\right|} \sum_{P\in\mathbb{P}_{2g+2}} \left(h_P R_P\right)^z = \left(\frac{q^{g+1}}{q-1}\right)^z \sum_{f \text{ monic}} \frac{d_z(f^2)}{|f|^2} \left(1+O\left(\frac{1}{g^{11}}\right)\right).
	\end{equation*}
\end{corollary}

Let $P\in\mathbb{P}_n$, $z\in\C$ such that $|z|\ll\log|P|/\left(\log_2|P|\ln\log_2|P|\right).$ Let  $Q$ represent an irreducible polynomial and define the generalized divisor function $d_z(f)$ on its prime powers as

\begin{equation*}
d_z(Q^a)=\frac{\Gamma(z+a)}{\Gamma(z)a!},
\end{equation*}
and extend it to all monic polynomials multiplicatively. We will prove the following lemmas which allow us to connect the complex moments of the random model to the complex moments of $L(1,\x_P).$

\begin{lemma}\label{l1}
	Let $P\in \mathbb{P}_n, N>4$ be fixed constant and $z\in\C$ such that $|z|\le \frac{\log|P|}{10 N \log_2|P|\ln\log_2|P|} $ and $M=N \log_2|P|$. Then 	
	\begin{equation*}
	L(1,\x_P)^z=\left(1+ O\left(\frac{1}{\left(\log|P|\right)^B}\right)\right) \sum_{\substack{f \text{ monic} \\ |f|\le|Q|^{1/3} \\ Q\mid f \Rightarrow \deg(Q)\le M}} \frac{\x_P(f)d_z(f)}{|f|},
	\end{equation*}
	where $B= N/2-2$.
\end{lemma}

Before giving the proof of the above, we state a few results.

\begin{lemma}\label{lnL(s,x)}
	Let $F$ be a monic polynomial, and $\x$ be a non-trivial character on $\left(A/AF\right)^\times.$ For a positive integer $M$ and any complex number $s$ with $\mathfrak{R}(s)=1$ we have	
	
	\begin{equation*}
	\ln L(s,\x) = - \sum_{\deg(P)\le M} \ln\left(1-\frac{\x(P)}{|P|^s}\right) + O\left(\frac{q^{(\frac{1}{2}-s)M}}{M} \deg(F)\right).
	\end{equation*}
\end{lemma}

\begin{proof}
	Recall that $
	L(s,\x)=\prod_{\substack{P \text{ Prime}}} \left(1-\x(P)/|P|^s \right)^{-1},$ 
	then 
	
	\begin{equation*}
	\begin{split}
	\ln L(s,\x) 
	=& - \sum_{\substack{P \text{ monic} \\ \text{irreducible} \\ \deg(P)\le M}} \ln\left(1-\frac{\x(P)}{|P|^s} \right) - \sum_{\substack{P \text{ monic} \\ \text{irreducible} \\ \deg(P) > M}} \ln\left(1-\frac{\x(P)}{|P|^s} \right).\\
	\end{split}
	\end{equation*} 
	We can see that the first term of our result already appears and we only need to bound the second sum. From the fact that $\log(1+x)=x+O(1)$ and $\left|\x(P)\right|\le 1$ and Proposition \ref{chi} we have that
	
	\begin{equation*}
	\begin{split}
	\sum_{\substack{P \text{ monic} \\ \text{irreducible} \\ \deg(P) > M}} \ln\left(1-\frac{\x(P)}{|P|^s} \right) &= \sum_{k=M}^\infty \sum_{\substack{P \text{ monic} \\ \text{irreducible} \\ \deg(P)=k}} \frac{\x(P)}{|P|^s} +O\left( \sum_{\substack{P \text{ monic} \\ \text{irreducible} \\ \deg(P) > M}} \frac{1}{|P|^s}  \right) \\
	&= \sum_{k=M}^\infty q^{-sk} \sum_{\substack{P \text{ monic} \\ \text{irreducible} \\ \deg(P)=k}} \x(P) +O\left( q^{(1-s)M}  \right) \\
	&\ll \deg(F) \sum_{k=M}^\infty q^{-sk} \frac{q^{\frac{k}{2}}}{k}\\
	&\ll \deg(F)  \frac{q^{(\frac{1}{2}-s)M}}{M},
	\end{split}
	\end{equation*} 
	with $F$ a non-perfect square.
\end{proof}

The next result is given below.

\begin{lemma}\label{Lz}
	Let $P\in \mathbb{P}_n, N>4$ be fixed constant and $z\in\C$ such that $|z|\le \frac{\log|P|}{10 N \log_2|P|\ln\log_2|P|} $ and $M=N \log_2|P|$. Then for $c_0$ some positive constant we have
	
	\begin{equation*}
	\sum_{\substack{f \text{ monic} \\ Q\mid f \Rightarrow \deg(Q)\le M}} \frac{\x_P(f)}{|f|} d_z(f) = \sum_{\substack{f \text{ monic} \\ |f|\le|P|^{1/3} \\ Q\mid f \Rightarrow \deg(Q)\le M}} \frac{\x_P(f)}{|f|} d_z(f) + O\left(|P|^{-\frac{1}{c_0\log_2|P|}}\right).
	\end{equation*}
\end{lemma}

\begin{proof}
	Let $z\in\C$ and $k\in\Z$ such that $|z|<k.$ Consider the sum

	\begin{equation*}
	\begin{split}
	\left|\sum_{\substack{f \text{ monic}\\ |f|>|P|^{1/3} \\ Q\mid f \Rightarrow \deg(Q)\le M}} \frac{\x_P(f)}{|f|} d_z(f) \right| &\ll \sum_{\substack{f \text{ monic} \\|f|>|P|^{1/3} \\ Q\mid f \Rightarrow \deg(Q)\le M}} \left|\frac{\x_P(f)}{|f|} d_z(f) \right| \\
	&\ll \sum_{\substack{f \text{ monic} \\|f|>|P|^{1/3} \\ Q\mid f \Rightarrow \deg(Q)\le M}} \frac{ d_k(f) }{|f|} ,
	\end{split}
	\end{equation*}
	since $\left|\x_P(f)\right|\le 1 $ and $d_z(f)<d_k(f)$ for $|z|<k.$ Let $0<\alpha\le\frac{1}{2}$ then using Rankin's trick we have

	\begin{equation*}
	\begin{split}
	\left|\sum_{\substack{f \text{ monic}\\ |f|>|P|^{1/3} \\ Q\mid f \Rightarrow \deg(Q)\le M}} \frac{\x_P(f)}{|f|} d_z(f) \right| 
	&\ll |P|^{-\frac{\alpha}{3}} \prod_{\substack{Q \text{ monic} \\ \text{irreducible} \\ \deg(Q)\le M}} \left(1- \sum_{j=1}^\infty \frac{ d_k(Q^j) }{|Q|^{(1-\alpha)j}}\right)\\
	&\ll |P|^{-\frac{\alpha}{3}} \exp\left(\sum_{\substack{Q \text{ monic} \\ \text{irreducible} \\ \deg(Q)\le M}} \sum_{j=1}^\infty \frac{ d_k(Q^j) }{|Q|^{(1-\alpha)j}}\right)\\
	&\ll |P|^{-\frac{1}{3M}} \exp\left(O\left(k \sum_{\substack{Q \text{ monic} \\ \text{irreducible} \\ \deg(Q)\le M}}  \frac{ 1}{|Q|}\right)\right)
	\end{split}
	\end{equation*}
	for $\alpha=1/M$ and $d_z(Q^r)=\Gamma(z+r)/\Gamma(z)r!$. Choose $M= N \log_2|P|$ and using Merten's Theorem, Lemma \ref{Mertens}, we have 

	\begin{equation*}
	\begin{split}
	\left|\sum_{\substack{f \text{ monic}\\ |f|>|P|^{1/3} \\ Q\mid f \Rightarrow \deg(Q)\le M}} \frac{\x_P(f)}{|f|} d_z(f) \right| &\ll |P|^{-\frac{1}{3M}} \exp\left(O\left(k \ln M\right)\right)\\
	&\ll |P|^{-\frac{1}{c_0 \log_2|P|}}.
	\end{split}
	\end{equation*}
\end{proof}

\textit{Proof of Lemma \ref{l1}}. Using Lemma \ref{lnL(s,x)} we can write

	\begin{equation*}
	\begin{split}
	L(1,\x_P)^z 
	& = \exp\left(- z \sum_{\deg(Q)\le M} \ln\left(1-\frac{\x_P(Q)}{|Q|}\right) + O\left( |z| \frac{q^{-\frac{M}{2}}}{M} \deg(P)\right)\right) \\
	& = \exp\left(- z \sum_{\deg(Q)\le M} \ln\left(1-\frac{\x_P(Q)}{|Q|}\right)\right) \exp\left(O\left( |z| \frac{q^{-\frac{M}{2}}}{M} \deg(P)\right)\right).
	\end{split}
	\end{equation*}

	Using the fact that $M=N \log_2|P|$ we have $q^{-\frac{M}{2}} = \left(\log|P|\right)^{-\frac{N}{2}},$ and deg$(P)=\log|P|, |z|\le \frac{\log|P|}{10 N \log_2|P|\ln\log_2|P|}$, so we can write the expression inside of the big Oh as	
	
	\begin{equation*}
	\frac{\left(\log|P|\right)^2}{\left(\log|P|\right)^{N/2}}\frac{1}{10 a \left(\log_2|P|\right)^2 \ln\log_2|P|}\ll \frac{1}{\left(\log|P|\right)^B}, 
	\end{equation*}
	since $N>4$. Hence,	
	
	\begin{equation*}
	\begin{split}
	L(1,\x_P)^z 
	& = \prod_{\substack{Q \text{ irreducible} \\ \deg(Q)\le M}}\left(\sum_{i=0}^\infty \frac{\x_P(Q^i)}{|Q|^i} d_z(Q^i)\right) \left( 1+ O\left(\frac{1}{\left(\log|P|\right)^B}\right)\right)\\
	& = \sum_{\substack{f \text{ monic} \\ Q\mid f \Rightarrow\deg(Q)\le M}}\left( \frac{\x_P(f)}{|f|} d_z(f)\right) \left( 1+ O\left(\frac{1}{\left(\log|P|\right)^B}\right)\right).
	\end{split}
	\end{equation*}	
	Applying Lemma \ref{Lz} the lemma follows.		\qed \\

Averaging $L(1,\x_P)$ over all $P\in\mathbb{P}_n$ making the use of Lemma \ref{l1} give us 

\begin{equation*}
\begin{split}
\sum_{P\in \mathbb{P}_n}L(1,\x_P)^z 
&= \left(1+O\left(\frac{1}{\left(\log|P|\right)^B}\right)\right) \sum_{\substack{f \text{ monic} \\ |f|\le |P|^{1/3} \\ Q\mid f \Rightarrow \deg(Q)\le M}} \frac{d_z(f)}{|f|}  \sum_{P\in \mathbb{P}_n}\x_P(f)\\
&= \left(1+O\left(\frac{1}{\left(\log|P|\right)^B}\right)\right) \left(S_1 +S_2\right),
\end{split}
\end{equation*} 
where

\begin{equation}\label{S1}
S_1 := \sum_{\substack{f \text{ monic and square} \\ |f|\le |P|^{1/3} \\ Q\mid f \Rightarrow \deg(Q)\le M}} \frac{d_z(f)}{|f|}  \sum_{P\in \mathbb{P}_n}\x_P(f),
\end{equation}
and  

\begin{equation}\label{S2}
S_2 := \sum_{\substack{f \text{ monic and not square} \\ |f|\le |P|^{1/3} \\ Q\mid f \Rightarrow \deg(Q)\le M}} \frac{d_z(f)}{|f|}  \sum_{P\in \mathbb{P}_n}\x_P(f).
\end{equation}

\subsection{Evaluating $S_2$: Contribution of the Non-Square Terms.}$\text{\color{white}djh}$\\

%

\begin{lemma}
	Let $P\in\mathbb{P}_n, N>4$ be a constant, $z\in\C$ be such that  $|z|\le \frac{\log|P|}{10N \log_2|P|\ln\log_2|P|},$ $k\in\Z$ with $|z|<k$ and $M=N \log_2|P|$. Then	
	
	\begin{equation*}
	S_2 \ll |P|^{\frac{1}{2}}\left(\log|P|\right)^k,
	\end{equation*}
	with $S_2$ defined as in (\ref{S2}).
\end{lemma}

\begin{proof}
	By Proposition \ref{chi} we have	

	\begin{equation*}
	\begin{split}
	S_2 &\ll \frac{q^{\frac{n}{2}}}{n} \sum_{\substack{f \text{ monic, }f\neq \Box \\|f|\le |P|^{1/3} \\ Q\mid f \Rightarrow \deg(Q)\le M}} \frac{d_z(f)}{|f|} \deg(f)\\
	& \ll \frac{q^{\frac{n}{2}}}{n} \sum_{j=0}^{[n/3]} q^{-j} j \sum_{\substack{f \text{ monic, } \\\deg(f)=j \\ Q\mid f \Rightarrow \deg(Q)\le M}} d_k(f)\\
	& \ll |P|^{\frac{1}{2}}\left(\log|P|\right)^k.
	\end{split}
	\end{equation*}	
\end{proof}

\subsection{Evaluating $S_1$: Contribution of the Square Terms.}$\text{\color{white}djh}$\\

%
Using the Prime Polynomial Theorem \ref{PNT} we have

\begin{equation*}
\begin{split}
S_1 &= \left( 1+ O\left(\frac{1}{\left(\log|P|\right)^B}\right)\right) \\
& \text{\color{white}dkmdh}\times\left(\sum_{\substack{f \text{ monic and square} \\ |f|\le |P|^{1/3} \\ Q\mid f \Rightarrow \deg(Q)\le M}} \frac{d_z(f)}{|f|} \left(\frac{|P|}{\log|P|} + O\left(\frac{|P|^{\frac{1}{2}}}{\log|P|}\right)\right)\right).
\end{split}
\end{equation*} 

Our goal in this section is to find an estimate of the above term, which is where the difficulty lies. So here enters the random model $L(1,\w)$ to help us to obtain the desired formula. Let $\left\{W_P\mid P \text{ prime}\right\}$ be the sequence defined in section \ref{random product}. In this section we prove the following Lemma.

\begin{lemma}\label{L1z}
	Let $P\in\mathbb{P}_n$. Let $z\in\C$ be such that $|z|\le \frac{\log|P|}{260 \log_2|P|\ln\log_2|P|}.$ Then	
	
	\begin{equation*}
	\frac{1}{|\mathbb{P}_n|} \sum_{P\in \mathbb{P}_n} L(1,\x_P)^z = E(L(1,\mathbb{X})^z) \left(1+O\left(\frac{1}{\left(\log|P|\right)^{11}}\right)\right).
	\end{equation*}
\end{lemma}

Recall Lemma \ref{E=dk}. Using the same reasoning as in the previous section we have for any $z\in\C$

\begin{equation*}
\begin{split}
E\left[L(1,\cdot)^z\right] 
& = \sum_{f \text{ monic}} \frac{d_z(f^2) }{|f|^2},
\end{split}
\end{equation*}
since $d_z(f)$ and $|f|$ can be seen as scalars and $L(1,\cdot)$ is defined in (\ref{lomega}). We have from the definition of random Euler product 

\begin{equation*}
E\left[L(1,\w)^z\right]= \prod_{\substack{P \text{ monic} \\ \text{irreducible}}} E\left[L_P(1,\w)^z\right],
\end{equation*}
where 

\begin{equation*}
\begin{split}
E\left[L_P(1,\w)^z\right] :& = E\left[\left(1- \frac{W_P}{|P|}\right)^{-z}\right]\\
& = \frac{1}{2} \left(\left(1-\frac{1}{|P|}\right)^{-z} + \left(1+\frac{1}{|P|}\right)^{-z}\right).
\end{split}
\end{equation*}
Writing the Taylor expansion for $\deg (P)>M$ we have that

\begin{equation*}
\left(1-\frac{1}{|P|}\right)^{-z} = 1+\frac{z}{|P|}+O\left(\frac{|z|}{|P|^2}\right),
\end{equation*}
and

\begin{equation*}
\left(1+\frac{1}{|P|}\right)^{-z} = 1-\frac{z}{|P|}+O\left(\frac{|z|}{|P|^2}\right).
\end{equation*}
Thus, for monic irreducible polynomial $Q$ with deg$Q>M$ we have 

\begin{equation*}
E\left[L_P(1,\w)^z\right]=1+O\left(\frac{|z|}{|Q|^2}\right),
\end{equation*}
and so

\begin{equation*}
\begin{split}
\prod_{\substack{P \text{ irreducible} \\ \deg(P)\ge M}} E\left[L_P(1,\w)^z\right]
&\ll \exp\left(|z| \sum_{\substack{P \text{ irreducible} \\ \deg(P)\ge M}} \frac{1}{|P|^2}\right) \\ 
&\ll \exp\left(|z|  \frac{1}{M^2}\right)\\
& = 1+O\left(\frac{1}{\left(\log|Q|\right)^B}\right).
\end{split}
\end{equation*}
The last equality follows from the relative size of $|z|$ and $M$ and for large enough $N$. Finally, from Lemma \ref{Lz} we have that 

\begin{equation*}
\begin{split}
E\left[L(1,\w)^z\right] &= \sum_{\substack{f \text{ monic}  \\ P\mid f \Rightarrow \deg(P)\le M}} \frac{d_z(f^2)}{|f|^2} \left(1+O\left(\frac{1}{\left(\log|Q|\right)^B}\right)\right)\\
&= \sum_{\substack{f \text{ monic}  \\ |f| <|Q|^{1/3} \\ P\mid f \Rightarrow \deg(P)\le M}} \frac{d_z(f^2)}{|f|^2} \left(1+O\left(\frac{1}{\left(\log|Q|\right)^B}\right)\right).
\end{split}
\end{equation*}

From the above and Lemma \ref{l1} and with the same choice made by Lumley \cite{Lum}, i.e.,  $N=26$ and $B=11$ we have proved Lemma \ref{L1z}. Using the fact that $|P|=q^n$ we obtain Theorem \ref{lumley theorem}. Corollaries \ref{lumley coral1} and \ref{lumley coral2} follows from the above discussion and equation (\ref{hP}).




\end{document}